\documentclass[12pt,letterpaper]{article}
\usepackage{amsthm,amssymb,amsmath}
\usepackage[pdftex]{hyperref}

\theoremstyle{plain}
\newtheorem{lemma}{Lemma}                    
\newtheorem{theorem}[lemma]{Theorem}         
\newtheorem{corollary}[lemma]{Corollary}     
\newtheorem{problem}[lemma]{Problem}         

\newtheorem{observation}[lemma]{Observation} 

\theoremstyle{definition}
\newtheorem{example}[lemma]{Example}         

\newcommand{\ggcqedsymbol}{$\square$}
\newcommand{\ggcqed}{\hbox{}\nobreak\hbox{\quad\ggcqedsymbol}}
\newcommand{\ggcendpf}{\ggcqed}
\newcommand{\ggcnopf}{\ggcqed}
\newcommand{\ggcendexample}{\ggcqed}

\newcommand{\defterm}[1]{\emph{#1}} 
\newcommand{\abstdefterm}[1]{#1} 

\newcommand{\ggcsetminus}{-}
\newcommand{\dom}{\ensuremath{\gamma}}
\newcommand{\domf}{\ensuremath{\gamma_f}}
\newcommand{\domg}{\ensuremath{\gamma_g}}

\date{January 17, 2017}

\title{Approximations of the Domination Number of a Graph}

\author{Glenn G.~Chappell\\
\small Department of Computer Science\\
\small University of Alaska\\
\small Fairbanks, AK 99775-6670\\
\small\texttt{chappellg{@}member.ams.org} \and
John Gimbel\\
\small Department of Mathematics and Statistics\\
\small University of Alaska\\
\small Fairbanks, AK 99775-6660\\
\small\texttt{jggimbel{@}alaska.edu} \and
Chris Hartman\\
\small Department of Computer Science\\
\small University of Alaska\\
\small Fairbanks, AK 99775-6670\\
\small\texttt{cmhartman{@}alaska.edu}}

\begin{document}

\maketitle
\centerline{\small \textit{2010 Mathematics Subject Classification.}
  05C69 (primary), 05C80 (secondary).}
\centerline{\small \textit{Key words and phrases.}
  domination, fractional domination, greedy domination, random graphs.}
\bigskip

\begin{abstract}
Given a graph $G$,
the \abstdefterm{domination number} $\dom(G)$ of $G$
is the minimum order of a set $S$ of vertices such
that each vertex not in $S$ is adjacent to some vertex in $S$.
Equivalently, label the vertices from $\{0,1\}$ so
that the sum over each closed neighborhood is at least one;
the minimum value of the sum of all labels, with this restriction,
is the domination number.
The \abstdefterm{fractional domination number} $\domf(G)$
is defined in the same way,
except that the vertex labels are chosen from $[0,1]$.
Given an ordering of the vertex set of $G$,
let $\domg(G)$ be the approximation of the
domination number by the standard greedy algorithm.
Computing the domination number is NP-complete;
however, we can bound $\dom$
by these two more easily computed parameters:
\[
\domf(G) \le \dom(G) \le \domg(G).
\]
How good are these approximations?
Using techniques from the theory of hypergraphs, one can show that,
for every graph $G$ of order $n$,
\[
\frac{\domg(G)}{\domf(G)} = O(\log n).
\]
On the other hand, we provide examples of graphs for which
$\dom / \domf = \Theta(\log n)$
and graphs for which
$\domg / \dom = \Theta(\log n)$.
Lastly, we use our examples to compare two bounds on $\domg$.
\end{abstract}

Graphs will be finite, simple, and undirected.
For a graph $G$,
we denote by $\delta(G)$ and $\Delta(G)$ 
the minimum and maximum degree of $G$, respectively.
We use $N[v]$ to denote the closed neighborhood of a vertex $v$.
The closed neighborhood of a sequence of vertices,
e.g., $N[v_1, v_2, \ldots, v_k]$,
is the union of the closed neighborhoods
of the vertices in the sequence.
We say that vertex $v$ \defterm{dominates} vertex $u$
if $u$ lies in the closed neighborhood of $v$.
See Haynes, Hedetniemi, \& Slater~\cite{HHS1998}
for definitions of graph-theoretic terms
and an introduction to domination in graphs.

If we assign weights to the vertices of a graph,
then the \defterm{total weight} of a set of vertices
is the sum of the weights of the vertices in the set.
We may consider a dominating set
as a $0,1$-weighting of the vertex set
so that the total weight of each closed neighborhood
is at least one.
Relaxing the requirement that the weights be integers,
we obtain a fractional version of the domination number.
Suppose we assign weight $f(v) \in [0,1]$ to each vertex $v$.
The function $f\colon V(G) \to [0,1]$
is a \defterm{fractional domination}
if for each vertex $v$,
\[
\sum_{u \in N[v]} f(u) \ge 1.
\]
The \defterm{fractional domination number} $\domf(G)$ of $G$ is
the minimum total weight of the vertex set,
taken over all fractional dominations of $G$.

A useful bound is the following,
which was discovered independently
by Grinstead \& Slater~\cite[Theorem~1]{GrSl1990}
and
by Domke, Hedetniemi, \& Laskar~\cite[Observation~3]{DHL1988}
(Observation~3 in the latter paper is slightly misstated,
with the inequalities in the wrong direction,
but the proof is correct).

\begin{lemma} \label{L:fracbounds}
For a graph $G$ of order $n$,
\[
\frac{n}{1+\Delta(G)} \le \domf(G) \le \frac{n}{1+\delta(G)}.\ggcnopf
\]
\end{lemma}

Throughout this paper,
we will implicitly assume an ordering on the vertex set of a graph.
Given such an ordering,
we can approximate the domination number using a greedy algorithm,
as follows.
Iteratively select vertices $x_1, x_2, \ldots, x_m$ so that,
for each $k = 1, 2, \ldots, m$,
vertex $x_k$ is chosen so that it dominates as many vertices of
$V(G)-N[x_1, x_2, ... , x_{k-1}]$
(that is, not-yet-dominated vertices) as possible.
Resolve ties by choosing $x_k$ as early as possible
in the ordering on $V(G)$.
Stop the iterative process
when every vertex is dominated by one of the $x_k$'s.
We refer to $x_1, x_2, \ldots, x_m$
as the \defterm{greedy dominating sequence}.
The \defterm{greedy domination number} $\domg(G) = m$
is the number of vertices in this sequence.

Determining the domination number of a general graph
is known to be NP-complete
(see Garey \& Johnson~\cite{GaJo1979});
it is natural to seek more easily computed approximations.
The values of $\domf$ and $\domg$ can be determined in polynomial time.
Further, the fact that $\dom$ lies in the interval $[\domf, \domg]$
follows easily from definitions.

\begin{observation} \label{O:dom3kinds}
For every graph $G$,
\[
\domf(G) \le \dom(G) \le \domg(G).\ggcnopf
\]
\end{observation}

We study the relationships of these three parameters further.

Techniques from the theory of hypergraphs
can be used to show that the ratio\break$\domg(G)/\domf(G)$ 
is $O(\log \Delta)$,
and thus $O(\log n)$, where $n$ is the order of $G$;
see Theorem~\ref{T:boundratio}, below.
Thus $\dom(G)$ must lie within a relatively small interval.
We produce examples showing that, asymptotically,
we can do no better.
We show that $\dom(G)/\domf(G)$ can be $\Theta(\log n)$,
and then we show that $\domg(G)/\dom(G)$ can be $\Theta(\log n)$.

Since $\domg$ is a useful upper bound on $\dom$, it is worthwhile to
consider upper bounds on $\domg$.
One such bound follows immediately from the above discussion:
\[
\domg(G) \le c \domf(G) \log n,
\]
for some constant $c$, where $n$ is the order of $G$.
Another class of bounds are those in which $\domg$
is bounded above by a constant multiple of $(n\log\delta)/\delta$.
The first of these was found by Alon \& Spencer~\cite{AlSp1992}
(see their Theorem~2.2 and the remarks following it).
A slightly improved bound was given by
Clark, Shekhtman, Suen, \& Fisher~\cite[Theorem~2]{CSSF1996};
we state this below.

\begin{theorem}[Clark, Shekhtman, Suen, \& Fisher~{\cite{CSSF1996}}]
  \label{T:logdd}
For every graph $G$ of order $n$,
\[
\domg(G)
  \le n\left[1-\prod_{i=1}^{\delta+1}\frac{i\delta}{i\delta+1}\right],
\]
where $\delta=\delta(G)$.\ggcnopf\end{theorem}

We note that the right side of the above inequality is
$\Theta\bigl([n \log \delta]/\delta\bigr)$.
At the conclusion of this paper,
we will compare these two bounds on $\domg$,
using examples to show that sometimes one is tighter,
and sometimes the other is.

\bigskip

In the following result,
we will use a concept dual to fractional domination.
A function $f\colon V(G) \to [0,1]$
is a \defterm{fractional packing} if
for each vertex $v$,
\[
\sum_{u \in N[v]} f(u) \le 1.
\]
Note that
the maximum total weight of $V(G)$, taken over all fractional packings,
and the minimum total weight of $V(G)$,
taken over all fractional dominations,
are described by dual linear programs
(see Haynes, Hedetniemi, \& Slater~\cite[Chapter~4]{HHS1998}
or
Domke, Hedetniemi, \& Laskar~\cite[Section~3]{DHL1988}).
Thus, by the principle of strong duality,
given a fractional packing on a graph $G$,
the total weight of the vertex set is at most $\domf(G)$.

We now prove an upper bound on $\domg(G)/\domf(G)$.
This is a special case of a more general result
on vertex covers of hypergraphs
and is similar to a bound found by Johnson~\cite[Theorem~4]{Joh1974}
and by Lov\'{a}sz~\cite[Corollary~2]{Lov1975}
(see also Schrijver~\cite[Theorem~77.2]{Sch2003}).

\begin{theorem} \label{T:boundratio}
For every graph $G$,
\[
\frac{\domg(G)}{\domf(G)}
  \le 1+\ln\bigl[1+\Delta(G)\bigr].
\]
\end{theorem}

\begin{proof}
Set $m = \domg(G)$.
Let $x_1, x_2, \ldots, x_m$ be the greedy dominating sequence.
For each vertex $v$ of $G$, let $g(v)$ be the
first vertex in the greedy dominating sequence that
dominates $v$.
Let $F(v)$ be the set of all vertices of $G$ that are
first dominated by $g(v)$;
that is, $F(v) = N[x_k] - N[x_1, x_2, \ldots, x_{k-1}]$,
where $x_k = g(v)$.
Let $w(v) = \frac{1}{\left|F(v)\right|}$.
So $w(v)$ is the reciprocal of the number of
vertices that are dominated
in the same step of the greedy algorithm as $v$.
Note that $\sum_{u \in F(v)} w(u) = 1$,
and thus $\sum_{v \in V(G)} w(v) = m$.

Our proof is based on that of Schrijver~\cite[Theorem~77.2]{Sch2003},
and proceeds as follows.
We assign weight $w(v)$ to each vertex $v$.
We find upper bounds on the weights of vertices
lying in a closed neighborhood,
and conclude that, if each vertex $v$ is given weight
$w(v) / \left(1+\ln\bigl[1+\Delta(G)\bigr]\right)$,
then the result is a fractional packing.
Applying linear programming duality, we then obtain
a lower bound on $\domf(G)$,
from which our result follows.

Let $v$ be a vertex of $G$.
We list the elements of $N[v]$ in the order
in which they were dominated in the greedy algorithm.
Letting $p = 1+\deg(v)$,
we represent $N[v]$ as $\{u_1, u_2, \ldots, u_p\}$,
where, if $g(u_i)$ comes before $g(u_j)$
in the greedy dominating sequence,
then $i < j$.

We claim that $w(u_i) \le \frac{1}{p+1-i}$ for each $u_i$.
Suppose for a contradiction
that $\left|F(u_i)\right| < p+1-i$, for some $u_i$.
Then
$\left|F(u_i)\right| < \bigl|\{u_i, u_{i+1}, \ldots, u_p\}\bigr|$,
and so replacing $g(u_i)$ by $v$ in the greedy dominating sequence
would increase the number of vertices dominated at this step
in the greedy algorithm.
However, this contradicts the definition of greedy dominating sequence,
and so
$\left|F(u_i)\right| \ge p+1-i$.
Thus,
\[
w(u_i) = \frac{1}{\left|F(u_i)\right|} \le \frac{1}{p+1-i},
\]
as claimed.

Hence, for each vertex $v$ we have
\[
\sum_{u \in N[v]} w(u)
  \le \sum_{i=1}^{p} \frac{1}{p+1-i}
  = \sum_{i=1}^{p} \frac{1}{i}
  \le 1+ \ln p
  \le 1 + \ln\bigl[1+\Delta(G)\bigr].
\]
Dividing by $1+\ln\bigl[1+\Delta(G)\bigr]$,
we obtain
\[
\sum_{u \in N[v]} \frac{w(u)}{1+\ln\bigl[1+\Delta(G)\bigr]}
  \le 1,
\]
and so assigning weight
$w(v) / \left(1+\ln\bigl[1+\Delta(G)\bigr]\right)$
to each vertex $v$,
results in a fractional packing.
Therefore, as noted before the statement of the theorem,
the sum of all vertex weights is bounded above by $\domf(G)$.
That is,
\[
\sum_{v \in V(G)} \frac{w(v)}{1+\ln\bigl[1+\Delta(G)\bigr]}
  \le \domf(G).
\]
Multiplying by $1+\ln\bigl[1+\Delta(G)\bigr]$,
we obtain
\[
\domg(G)
  = m
  = \sum_{v \in V(G)} w(v)
  \le \left(1+\ln\bigl[1+\Delta(G)\bigr]\right) \domf(G).
\]
Dividing by $\domf(G)$ yields our result.\ggcendpf\end{proof}

Hence the following.

\begin{corollary} \label{C:domupperbds}
For any graph $G$ of order $n$ with maximum degree $\Delta \ge 2$
\[
\dom(G) \le c_1 \ln(\Delta) \domf(G)
\]
and
\[
\dom(G) \le c_2 \ln(n) \domf(G),
\]
where $c_1$ and $c_2$
are appropriately chosen constants.\ggcnopf\end{corollary}

The preceding theorem and corollary
place restrictions on the value of $\dom$.
We now show that these restrictions are asymptotically best possible
up to a constant factor.
We begin with a construction of a family of graphs in which $\dom$ lies
near the high end of the interval $[\domf, \domg]$.
Later, we will obtain better results using random graphs.

\begin{example} \label{E:torusconstruct}
Given a positive integer $t$,
we construct a graph $J_t$ of order $n = (2t)^{2t-1}$
such that
\[
\domf(J_t) = e+o(1) = \Theta(1)
\]
and
\[
\dom(J_t) = 2t
  = \Theta\left(\frac{\log n}{\log \log n}\right).
\]

Let $t$ be a positive integer.
Set $d = 2t-1$ and $n = (2t)^d$.
Let $G$ be the graph $K_{2t} \ggcsetminus tK_2$,
that is, $K_{2t}$ with a perfect matching removed.
Let $J_t$ be the graph whose vertices are $d$-tuples of the form
$( x_1, x_2, \ldots, x_d )$ where each $x_i$ is a vertex in $G$.
Let vertices $( x_1, x_2, \ldots, x_d )$
and $( y_1, y_2, \ldots, y_d )$ be
adjacent in $J_t$ if for each $i$,
the vertices $x_i$ and $y_i$ are equal or adjacent in $G$.
(The way in which $J_t$ is constructed from $G$ is often called
the ``strong [direct] product''.)
We note that $J_t$ has order $n$.

We show that $J_t$ has the required properties.
For each vertex $v$ of $G$,
denote by $\overline{v}$ the unique vertex in $G$
that is not adjacent to $v$.

Note that $J_t$ is regular of degree $(2t-1)^d - 1$.
By Lemma~\ref{L:fracbounds},
\[
\domf(J_t) = \frac{n}{(2t-1)^d}
  = \frac{(d+1)^d}{d^d} = e+o(1).
\]

Let $S$ be a set of $d$ vertices of $J_t$.
We write
$S = \bigl\{\,\left(x_1^i, x_2^i, \ldots, x_d^i\right)
  \mathrel{\big|}
  i = 1, 2, \ldots, d \,\bigr\}$.
Let $u = \left(\overline{x_1^1}, \overline{x_2^2}, \ldots,
  \overline{x_d^d}\right)$.
Then $u$ is not dominated by any vertex in $S$,
so $S$ is not a dominating set.
Hence, the domination number of $J_t$ is at least $d+1$.
Now let $A$ be the set of all vertices in $J_t$
of the form $(v,v,v, \ldots , v)$
where $v$ is a vertex in $G$.
Since there are $d+1$ such vertices, but only $d$ coordinates,
every vertex of $J_t$ must be dominated by at least one vertex of $A$.
Thus, $A$ is a dominating set of size $d+1$,
and so $\dom(J) = d+1 = 2t$.\ggcendexample\end{example}

For the graph $J_t$ of Example~\ref{E:torusconstruct},
$\dom/\domf = \Theta(\log n / \log \log n)$.
Thus we have constructed an infinite family of graphs for which the
ratio $\dom/\domf$ is unbounded.
However, the ratio is not as high as we would like.
Using random graphs, we can produce better examples,
for which $\dom/\domf$ is, with high probability,
$\Theta(\log n)$.

For each natural number $n$,
let $R_n$ be a random graph on $n$ labeled vertices
with edge probability $1/2$.
Given a graphical property $P$
we say that $R_n$ \defterm{almost surely} (a.s.) has $P$
if the probability that $R_n$ has $P$ goes to one
as $n$ approaches infinity.
See Palmer~\cite{Pal1985} for an introduction to random graphs.

It is known that the domination number of $R_n$
is almost surely $\Theta(\log n)$
(see Nikoletseas \& Spirakis~\cite[Lemmas~1 \&~2]{NiSp1994}).
In fact, much stronger results are known.
Weber~\cite[Theorem~2]{Web1981}
showed that $\dom(R_n)$ is a.s.\ equal to one of
two values given by explicit formulae.
For our purposes,
it suffices that $\dom(R_n)$
is a.s.\ $\Theta(\log n)$.
On the other hand, $\domf(R_n)$ is a.s.\ $\Theta(1)$.
We give a short proof of these facts below.

\begin{theorem} \label{T:rand}
Almost surely,
\[
\domf(R_n) = 2+o(1)
\]
and
\[
\dom(R_n) = \Theta(\log n).
\]
\end{theorem}

\begin{proof}
It is known that the degrees of all vertices in $R_n$
tend to concentrate tightly around $n/2$.
In particular, a.s.
\[
\bigl[1-o(1)\bigr]\frac{n}{2}
  \le \delta(R_n)
  \le \Delta(R_n)
  \le \bigl[1+o(1)\bigr]\frac{n}{2}.
\]
This follows from a Chernoff bound~\cite[Theorem~1]{Che1952};
for a proof,
see Palmer~\cite[Theorem~5.1.4]{Pal1985}.
Applying Lemma~\ref{L:fracbounds},
we conclude that a.s.\ $\domf(R_n) = 2 + o(1)$.

Since $\domf(R_n)$ is a.s.\ $\Theta(1)$,
by Corollary~\ref{C:domupperbds}
we see that
$\dom(R_n)$ is a.s.\ $O(\log n)$.
It remains to show that $\dom(R_n)$ is a.s.\ $\Omega(\log n)$.
Fix $\varepsilon$ with $0 < \varepsilon < 1$.
Set $p=\bigl\lfloor(1-\varepsilon) \log_2 n\bigr\rfloor$.
We show that a.s.\ $\dom(R_n) > p$, which will complete our proof.

Our argument is similar to that given by
Nikoletseas \& Spirakis~\cite[Lemma~1]{NiSp1994}.
Let $S$ be a subset of $V(G)$ with order $p$.
If $v$ is a vertex not in $S$ then
the probability that $S$ dominates $v$
is $1-\left(\frac{1}{2}\right)^p$.
Hence, the probability that $S$ dominates $R_n$
is $\left[1-\left(\frac{1}{2}\right)^p\right]^{n-p}$.
Let $E$ be the expected number of $p$-sets that dominate $R_n$.
Then,
\begin{align*}
E = \binom{n}{p} \left[1-\left(\frac{1}{2}\right)^p\right]^{n-p}
  &\le n^p \left[e^{-\left(1/2\right)^p}\right]^{n-p} \\
  &\le n^p e^{-\left(1/n^{1-\varepsilon}\right)(n-p)}
    \qquad\text{since $2^p\le n^{1-\varepsilon}$} \\
  &=   e^{p \ln(n)-\left(n/n^{1-\varepsilon}\right)} e^{p/n^{1-\varepsilon}} \\
  &\le c e^{p \ln(n) - n^\varepsilon},\\
\end{align*}
for some constant $c$.
But the last expression goes to zero
as $n$ approaches infinity.
Hence, $R_n$ a.s.\ has no dominating $p$-set.
This leads to the desired result.\ggcendpf\end{proof}

When the random graph $R_n$ almost surely has some property,
we may conclude that,
for each sufficiently large $n$,
there exists a graph of order $n$
having the property.
Hence, we obtain the following.

\begin{corollary} \label{C:domhi}
There exist graphs $G_n$,
for infinitely many integers $n$,
so that each $G_n$ has order $n$, and
\[
\frac{\dom(G_n)}{\domf(G_n)} = \Theta(\log n).\ggcnopf
\]
\end{corollary}

Thus, the bounds in Corollary~\ref{C:domupperbds}
are asymptotically best possible.
We have proven this using probabilistic methods.
The best explicit construction we have been able to find
is that of the graphs $J_t$ from Example~\ref{E:torusconstruct},
for which the ratio $\dom/\domf$ is smaller:
$\Theta(\log n/\log \log n)$.
We ask whether an explicit construction can be found
for the larger ratio.

\begin{problem} \label{B:construct}
Find an explicit construction of graphs $G_n$,
for infinitely many integers $n$,
so that each $G_n$ has order $n$, and
\[
\frac{\dom(G_n)}{\domf(G_n)} = \Theta\left(\log n\right).\ggcnopf
\]
\end{problem}

We have seen that $\domg/\domf$ is $O(\log n)$,
and that the ratio $\dom/\domf$ may be $\Theta(\log n)$.
In our next example
the ratio $\domg/\dom$ is $\Theta(\log n)$.
Thus, $\dom$ is near the low end of the interval $[\domf, \domg]$,
and the greedy algorithm
approximates the domination number relatively poorly.

\begin{example} \label{E:cliqueconstruct}
Given an integer $t \ge 4$,
we construct a graph $H_t$ of order $n = 2^{t+2}$ such that
\[
\domf(H_t) = \dom(H_t) = 4
\]
and
\[
\domg(H_t) = t
  = \Theta(\log n).
\]

Let $t\ge 4$ be a natural number.
Let $u_1$, $u_2$, $u_3$, $u_4$ be vertices
and set $S = \{u_1, u_2, u_3, u_4 \}$.
To construct $H_t$,
begin with the union of $S$ and $t$ disjoint cliques:
\[
S \cup \bigl[K_4 \cup K_8 \cup K_{16} \cup
  \cdots \cup K_{2\cdot2^t}\bigr].
\]
Add additional edges so that each vertex of $S$ is adjacent
to one quarter of the vertices in each clique,
and no two vertices of $S$ have any common neighbors.
Let $H_t$ be the resulting graph.
We note that the order of $H_t$ is
\[
4 + 4 \bigl[1 + 2 + 4 + \cdots + 2^{t-1} \bigr] = 2^{t+2}.
\]

Given a fractional domination of $H_t$,
the total weight of the vertices in each $N[u_i]$
is at least 1.
Since the sets $N[u_1]$, $N[u_2]$, $N[u_3]$, $N[u_4]$ are pairwise disjoint,
we have $\domf(H_t) \ge 4$.
On the other hand, $S$ dominates $H_t$,
so $\dom(H_t) \le 4$.
Thus,
\[
4 \le \domf(H_t) \le \dom(H_t) \le 4,
\]
and we have $\domf(H_t) = \dom(H_t) = 4$.

If we approximate $\dom(H_t)$ with the greedy algorithm,
then we will never choose any vertex in $S$.
The greedy dominating sequence will contain
one vertex from each of the cliques used to construct $H_t$.
Since $t \ge 4$ the first four such vertices chosen will dominate
the four vertices in $S$, and so $\domg(H_t) = t$.\ggcendexample\end{example}

Letting $n = 2^{t+2}$,
and letting $G_n$ be $H_t$ from the above example,
we obtain the following.

\begin{corollary} \label{C:domlo}
There exist graphs $G_n$,
for infinitely many integers $n$,
so that each $G_n$ has order $n$, and
\[
\frac{\domg(G_n)}{\dom(G_n)} = \Theta(\log n).\ggcnopf
\]
\end{corollary}

We now consider upper bounds on $\domg$.
By Theorem~\ref{T:boundratio} we have, for a graph $G$ of order $n$,
\begin{equation} \label{Eq:ratiobound}
\domg(G) \le c_1 \domf(G) \log n,
\end{equation}
for some constant $c_1$.
And by Theorem~\ref{T:logdd}, we have
\begin{equation} \label{Eq:clarkbound}
\domg(G) \le c_2\frac{n\log\delta(G)}{\delta(G)},
\end{equation}
for some constant $c_2$.

Consider these bounds
for the graph $H_t$ from Example~\ref{E:cliqueconstruct}.
We have $\domf(H_t) = 4$, and clearly $\delta(H_t) = 4$.
Thus, letting $n$ be the order of $H_t$,
the right-hand side of~(\ref{Eq:ratiobound}) is $\Theta(\log n)$,
while the right-hand side of~(\ref{Eq:clarkbound}) is $\Theta(n)$,
making~(\ref{Eq:ratiobound}) by far the tighter bound.

On the other hand, let $t$ be a positive integer,
and let $G$ be a $t$-clique with a pendant vertex joined to each
clique vertex
(a ``hairy clique'').
Letting $n$ be the order of $G$,
we have $\domf(G) = t = n/2$, and $\delta(G) = 1$.
Thus,
the right-hand side of~(\ref{Eq:ratiobound}) is $\Theta(n \log n)$,
while the right-hand side of~(\ref{Eq:clarkbound}) is $\Theta(n)$,
making~(\ref{Eq:clarkbound}) the tighter bound.

\end{document}